\documentclass[11pt]{article}
\usepackage{amssymb,amsmath,latexsym}
\usepackage{amsthm}
\usepackage[shortlabels]{enumitem}
\usepackage{caption}
\usepackage{xcolor}
\usepackage{float}
\usepackage{graphicx}
\usepackage{arydshln}
\usepackage[utf8]{inputenc}
\usepackage[russian,english]{babel}
\usepackage[T2A]{fontenc}
\usepackage{tikz}
\usepackage{url}
\usepackage{breakurl}
\usepackage[colorlinks=true,citecolor=blue,linkcolor=blue,urlcolor=blue,bookmarks,bookmarksopen,bookmarksdepth=2,backref=page,breaklinks]{hyperref}

\usepackage{bookmark}
\bookmarksetup{
	numbered, 
	open,
}

\renewcommand*{\backref}[1]{}
\renewcommand*{\backrefalt}[4]{%
	\ifcase #1 (Not cited.)%
	\or        (Cited on page~#2.)%
	\else      (Cited on pages~#2.)%
	\fi}

\newcommand{\D}{\mathbb D}
\newcommand{\E}{\mathbb E}
\newcommand{\M}{\mathbb M}

\newcommand{\N}{\mathbb N}
\newcommand{\Z}{\mathbb Z}

\renewcommand{\S}{\mathbb S}
\newcommand{\F}{\mathbb F}

\renewcommand{\P}{{\mathbb P}}

\DeclareMathOperator{\GL}{GL}

\theoremstyle{plain}
\newtheorem{theorem}{Theorem}[section]
\newtheorem{lemma}[theorem]{Lemma}
\newtheorem{proposition}[theorem]{Proposition}

\theoremstyle{definition}

\newtheorem{definition}[theorem]{Definition}

\newtheorem{notation}[theorem]{Notation}

\numberwithin{theorem}{section}
\numberwithin{equation}{section}
\numberwithin{table}{section}
\numberwithin{figure}{section}
\DeclareMathOperator{\End}{End}
\DeclareMathOperator{\Aut}{Aut}

\begin{document}
\title{The autotopism group of a family of commutative semifields}

\author{Lukas K\"olsch$^1$, Alexandra Levinshteyn$^2$, and Milan Tenn$^3$\vspace{0.4cm} \ \\
$^1$ University of South Florida \\\tt lukas.koelsch.math@gmail.com\vspace{0.4cm}\\
$^2$ University of Illinois, Urbana-Champaign \\ \tt alexandraplevinshteyn@gmail.com \vspace{0.4cm}\\
$^3$ Swarthmore College \\\tt mtenn1@swarthmore.edu\vspace{0.4cm}
}

\date{\today}
\maketitle
\abstract{
We completely determine the autotopism group of the (as of now) largest family of commutative semifields found by G\"olo\u{g}lu and K\"olsch. Since this family of semifields generally does not have large nuclei, this process is considerably harder than for families considered in preceding work. Our results show that all autotopisms are semilinear over the degree 2 subfield and that the autotopism group is always solvable. Using known connections, our results also completely determine the automorphism groups of the associated rank-metric codes and the collineation groups of the associated translation planes.}\ \\[2mm]


\thispagestyle{empty}

\section{Introduction}
A finite \textbf{semifield} $\S = (S,+,\circ)$ is a finite set $S$ equipped with two operations $(+,\circ)$
satisfying the following axioms. 
\begin{enumerate}
\item[(S1)] $(S,+)$ is a group.
\item[(S2)] For all $x,y,z \in S$,
\begin{itemize}
\item $x\circ (y+z) = x \circ y + x \circ z$,
\item $(x+y)\circ z = x \circ z + y \circ z$.
\end{itemize}
\item[(S3)] For all $x,y \in S$, $x \circ y = 0$ implies $x=0$ or $y=0$.
\item[(S4)] There exists $\epsilon \in S$ such that $x\circ \epsilon = x = \epsilon \circ x$.
\end{enumerate}
An algebraic object satisfying the first three of the above axioms is called 
a \textbf{pre-semifield}. While infinite semifields exist, we are in this paper only concerned with the finite case and thus all (pre)-semifields are always assumed to be finite. If $\P = (P,+,\circ)$ is a pre-semifield, then $(P,+)$ is 
an elementary abelian $p$-group \cite[p. 185]{Knuth65}, and $(P,+)$ can be viewed as
an $n$-dimensional $\F_{p}$-vector space $\F_p^n$. If $\circ$ is commutative, we call $\S$ a commutative semifield. 
Note that if $\circ$ is associative then $\S$ is the finite field $\F_{p^n}$ by 
Wedderburn's theorem which states that a finite division ring is a field. 
A pre-semifield $\P = (\F_p^n,+,\circ)$ 
can be converted to a semifield $\S = (\F_p^n,+,\ast)$ using {\em Kaplansky's trick} 
by defining the new multiplication as
\[
	(x \circ e) \ast (e \circ y) = x \circ y,
\]
for any nonzero element $e \in \F_p^n$, making $(e \circ e)$ the multiplicative 
identity of $\S$. Pre-semifields have received a lot of attention from geometers since they can be used to construct non-desarguesian planes, more precisely, translation planes whose duals are also translation planes, see e.g.~\cite{dembowski1997finite, hughesbook} for a comprehensive overview of the connections of semifields to finite geometry. 

A pre-semifield is an $\F_{p}$-algebra, thus the multiplication
is bilinear. Therefore we have $\F_{p}$-linear left and right multiplications $L_x,R_y : \F_p^n \to \F_p^n$, with
\[
L_x(y) := x \circ y =: R_y(x).
\]
The mapping $L_x$ (resp. $R_y$) is a bijection whenever $x \ne 0$ (resp. $y \ne 0$) 
by (S3). By fixing a basis of $\F_p^n$, we can view $L_x$ and $R_y$ as matrices over $\F_p$, i.e., as elements in $M^{n\times n}(\F_p)$. So
\[\mathcal{C}=\{R_y(x) \colon y \in \F_p^n\}\]
is an additively closed set where all non-zero matrices are invertible. Such a set is called a (semifield) \textbf{spread set}. It is also an optimal linear rank-metric code (in  the sense that it satisfies the Singleton bound with equality) for specific parameters, see~\cite{sheekey}. 

One of the most important properties of (pre)-semifields is their \textbf{autotopism group}. 

The autotopism group of a semifield is isomorphic to the automorphism group of the corresponding spread set, which we define now:

\begin{definition}
    The automorphism group $\Aut(\mathcal{C})$ of a spread set $\mathcal{C}\leq M^{n\times n}(\F_p)$ is 
    \[\Aut(\mathcal{C})=\left\{\right(X,Y)\in \GL(n,p)\times \GL(n,p) \colon X\mathcal{C}Y= \mathcal{C}\}.\]
\end{definition}
 The autotopism group of a semifield can for instance be used to show that semifields are inequivalent, see e.g.~\cite{golouglu2023exponential}. In the geometric setting, determining the autotopism group of a (pre)-semifield is equivalent to determining the collineation group of the corresponding translation plane, see~\cite{dembowski1997finite}, and in the coding theory setting, it is equivalent to the automorphism group of the rank-metric code.
 
Certain subgroups of the automorphism group play an important role (see e.g.~\cite{MP12}).
\begin{definition}
    Let $\mathcal{C}(\S)$ be the semifields spread set belonging to a pre-semifield $\S=(\F_p^n,+,\circ)$. Define the sets
    \[\mathcal{N}_m(\S)^*=\left\{ Y\in \GL(n,p) \colon \mathcal{C}Y= \mathcal{C}\right\}, \mathcal{N}_r(\S)^*=\left\{X\in \GL(n,p) \colon X\mathcal{C}=\mathcal{C}\right\}.\]
    Then 
    $\mathcal{N}_m(\S):=\mathcal{N}_m(\S)^*\cup \{0_n\}$ and $\mathcal{N}_r(\S):=\mathcal{N}_r(\S)^*\cup \{0_n\}$
    are called the middle and right nucleus of $\S$, respectively. Here, $0_n \in M^{n\times n}(\F_p)$ is the zero matrix.
\end{definition}
Clearly, $\mathcal{N}_m(\S)^*$ and $\mathcal{N}_r(\S)^*$ are normal subgroups of the automorphism group. The nuclei are always finite fields. It is follows immediate from the definition of the nuclei (see~\cite[Theorem 3.4.]{liebhold2016automorphism}) that
\[\mathcal{N}_r(\S)^* \times \mathcal{N}_m(\S)^* \leq \Aut(\mathcal{C}(\S)) \leq N(\mathcal{N}_r(\S)^*) \times N(\mathcal{N}_m(\S)^*),\]
where $N(S)$ denotes the normalizer of the set $S$ in $\GL(n,\F_p)$. This indicates the following rule of thumb: \emph{The bigger the nuclei of $\S$, the easier it is to determine autotopism group of semifield.} Indeed, if one of the nuclei is very large, computing the autotopism group is comparatively easy, see e.g.~\cite{hui2015autotopism}, where the autotopism group of the Dickson semifields of size $p^{2m}$ is determined, using the fact that one of the nuclei has size $p^m$. Generally, it is however very difficult to determine the autotopism group of a semifield, and for most known families of semifields with small nuclei, the autotopism group is either not known, or determining the autotopism group takes considerable effort, like in the case of the cyclic semifields~\cite{dempwolff2011autotopism}.\\

\textbf{Our contribution.} In this paper, we will determine the autotopism group of the following family of commutative semifields, found in~\cite{golouglu2023exponential}, that is (as of now) the largest known family of commutative semifields. 

\begin{theorem}[\cite{golouglu2023exponential}] \label{thm:SF}
Define a binary operation $\circ$ on $\F_{p}^{2m}\cong \F_{p^m}\times \F_{p^m}$ via 
\begin{equation} \label{eq:SF}
    (x,y) \circ (u,v) = \left(x^qu+xu^q+B(y^qv+ yv^q),x^rv+yu^r+A(yv^r+y^ru)\right)
\end{equation}
where $p$ is odd, $q=p^k$ for some $1 \leq k \leq m-1$, $r=p^{k+m/2}$, $B \in \F_{p^m}$ is a non-square, $A \in \F_{p^m}$ such that $AB \in \F_{p^{m/2}}^\times$, and $m/\gcd(k,m)$ is odd. Then $\S=(\F_{p}^{2m}\cong \F_{p^m}\times \F_{p^m},+,\circ)$ is a pre-semifield.
\end{theorem}
Note that this implies that $A$ is a non-square in $\F_{p^m}$.
Consequently, the semifield spread set $\mathcal{C}(\S)$ is comprised of the mappings
\begin{equation} \label{eq:spreadset}
    R_{u,v} : (x,y) \mapsto (R^{(1)}_{u,v}(x,y),R^{(2)}_{u,v}(x,y)),
\end{equation}

with
\begin{equation}\label{eq:spreadset2}
    	R^{(1)}_{u,v}(x,y) = x^qu+xu^q+B(y^qv+yv^q) \textrm{ and } R^{(2)}_{u,v}(x,y) =  x^rv+Axv^r+Ay^ru+yu^r.
\end{equation}
 Then 
\[\mathcal{C}(\S)=\{ R_{u,v} : (u,v) \in \F_{p^m} \times \F_{p^m}\}.\]
The nuclei were determined in~\cite[Theorem 7.2.]{golouglu2023exponential}:
\begin{theorem} \label{thm:nuclei}
    Let $\S=(\F_{p^m}\times \F_{p^m},+,\circ)$ with the multiplication as in Eq.~\eqref{eq:SF}. 
    Then 
    \begin{align*}
    \mathcal{N}_r(\S)&=\{(x,y)\mapsto (ax,ay) \colon a \in \F_{p^{\gcd(k+m/2,m)}} \}, \\
    \mathcal{N}_m(\S)&=\{(x,y)\mapsto (ax,ay) \colon a \in \F_{p^{2\gcd(k+m/2,m)}} \}.
    \end{align*} 
\end{theorem}
In particular, it is possible to choose $k,m$ in a way that $\mathcal{N}_r(\S)$ is the smallest possible (i.e., isomorphic to the prime field $\F_p$) and $\mathcal{N}_m(\S) \cong \F_{p^2}$. Thus, the techniques that simplify computing the autotopism group based on using large nuclei do not apply. 

The main result of this paper is:
\begin{theorem}[Main result] \label{thm:main}
    Let $\S$ be a pre-semifield with multiplication as defined in Theorem~\ref{thm:SF} and let $e=\gcd(k,m)$. The autotopism group of $\S$ (or equivalently, the automorphism group of its spread set $\Aut(\mathcal{C}(\S))$) is isomorphic to
    \begin{enumerate}
        \item $\Aut(\mathcal{C}(\S)) \cong \Z_{p^m-1}\times \Z_{p^e-1}$, if $A^2$ is not an $(r-1)$-st power. 
        \item $\Aut(\mathcal{C}(\S)) \cong (\Z_{p^m-1} \times \Z_{p^e-1}\times \Z_2) \rtimes \Z_{m/i}$, if $A^2$ is an $(r-1)$-st power and $i$ is the smallest positive integer such that $A^{p^i-1}$ is an $(r-1)$-st power.
    \end{enumerate}
    Moreover, $\Aut(\mathcal{C}(\S)) \leq \Gamma L(2,\F_{p^m})$ in both cases.
\end{theorem}
In fact, we determine all automorphisms explicitly --- the precise description can be found for the first case in Theorem~\ref{thm:diagonalcase} for $i=0$. The second case is covered by Theorem~\ref{thm:diagonalcase} together with Theorem~\ref{thm:antidiagonalcase}.

We want to note that even though the nuclei are possibly small, the autotopism group of these semifields is always large. This in itself is also unusual among the known semifields, see e.g.~\cite{dempwolff2011autotopism,hui2015autotopism}, where the autotopism groups of semifields were computed with the result that the entire autotopism group contains the group generated by the nuclei as a small index subgroup. 

The rest of the paper is devoted to prove Theorem~\ref{thm:main}. 

\section{Notation, overview, and preliminaries}
We fix some information we will use throughout the paper, following the notation of~\cite{golouglu2023exponential} that introduced the pre-semifields defined by Eq.~\eqref{eq:SF}.\\
\noindent\begin{minipage}{0.35\textwidth}
\begin{tikzpicture}
    \node (Q1) at (1,0) {$\F_{p}$};
    \node (Q2) at (2,2) {$\D = \F_{p^d}$};
    \node (Q3) at (3,4) {$\E = \F_{p^e}$};
    \node (Q4) at (0,3) {$\F_{p^{m/2}}$};
    \node (Q5) at (1,5) {$\M = \F_{p^m}$};
    \node (Q6) at (2,7) {$\F_{p^{2m}}$};

    \draw[dotted] (Q1)--(Q2) node[pos=0.7, below, inner sep=0.25cm] {$d$};
    \draw (Q2)--(Q3) node[pos=0.7, below, inner sep=0.25cm] {$2$}; 
    \draw (Q2)--(Q4);
    \draw (Q3)--(Q5) node[pos=0.3, above, inner sep=0.25cm] {$\frac{m}{e}$};
    \draw (Q4)--(Q5) node[pos=0.3, above, inner sep=0.25cm] {$2$};
    \draw (Q5)--(Q6) node[pos=0.3, above, inner sep=0.25cm] {$2$}; 
\end{tikzpicture}
\end{minipage}
\begin{minipage}{0.65\textwidth}
\begin{notation} \label{notation}
\begin{itemize}
\setlength\itemsep{0.3em}
\item[]
\item $p$ is an odd prime.
\item $n = 2m$, $m$ is even.
\item $Q = p^{m/2}$, \quad $Q^2 = p^m$.
\item $q = p^k$, \quad $r = p^{k+m/2} = Qq$ \ \ with $1 \le k \le m-1$. 
\item $\S_{q,A,B}=(\M\times \M,+,\circ)$ is a pre-semifield defined by Eq.~\eqref{eq:SF} with parameters $q,A,B$. $B$ is a non-square in $\F_{p^m}$ and $AB \in \F_Q^\times$.
\item $\mathcal{C}(\S_{q,A,B})$ is the spread set of $\S_{q,A,B}$ containing the mappings as defined in Eqs.~\eqref{eq:spreadset} and~\eqref{eq:spreadset2}.
\item $e = \gcd(k,m)$ with $m/e$ odd. 
\item $d = \gcd(k+m/2,m)$.
\item $e = 2d$ by Lemma \ref{lem_thm}.
\item $\E = \F_{q} \cap \M = \F_{q^2} \cap \M = \F_{r^2} \cap \M$ by Lemma \ref{lem_thm}.
\item $\D = \F_{r} \cap \M$.
\end{itemize}\vspace{5mm}
\end{notation}
\end{minipage}

We need some technical lemmas that will be used throughout the paper. The first one is basic and well known, the second one was proven in~\cite{golouglu2023exponential}.
\begin{lemma} \label{lem:gcd}
	Let $i,m \in \N$ and $p$ be a prime. Then
	\begin{itemize}
	\item $\gcd(p^i-1,p^m-1) = p^{\gcd(i,m)}-1$.
	\item $\gcd(p^i+1,p^m-1)=\begin{cases}
		1 & \text{if } m/\gcd(i,m) \text{ odd, and } p=2, \\
		2 & \text{if } m/\gcd(i,m) \text{ odd, and } p>2, \\
		p^{\gcd(i,m)}+1 & \text{if } m/\gcd(i,m) \text{ even}.
	\end{cases}$
\end{itemize}
\end{lemma}
	
\begin{lemma} \cite[Lemma 4.3]{golouglu2023exponential}\label{lem_thm}
We have
\begin{enumerate}[(i)]
\item $(-1) \not\in (\M^\times)^{q-1}$.
\item $\E = \F_{q} \cap \M = \F_{q^2} \cap \M = \F_{r^2} \cap \M$.
\end{enumerate}
\end{lemma}

The following is an adaption of Theorem~\cite[Theorem 2.9.]{MP12}
\begin{theorem} \label{thm:semilinear}
    Using Notation~\ref{notation}, let $\S:=\S_{q,A,B}$.
    Assume that $X,Y \in \GL(n,p)$ and $X\Aut(\mathcal{C}(\S))Y^{-1}=\Aut(\mathcal{C}(\S))$. Then $Y$ is semilinear over $\E$ and $X$ is semilinear over $\D$.
\end{theorem}
\begin{proof}
    By~\cite[Corollary 2.3.]{MP12}, we get that $Y\mathcal{N}_m(\S)Y^{-1}=\mathcal{N}_m(\S)$ and $X\mathcal{N}_r(\S)X^{-1}=\mathcal{N}_r(\S)$. By Theorem~\ref{thm:nuclei}, ${N}_m(\S)=\{f_a \colon (x,y)\mapsto a\cdot (x,y) \mid a \in \E\}$ so $f_a \mapsto Yf_aY^{-1}$ is an isomorphism on the (field) ${N}_m(\S)$, so it necessarily has the structure of a Frobenius automorphism, $Yf_aY^{-1}=f_{a^{p^i}}$ for some $i$ and for all $a \in \E$. In particular, $Y$ is semilinear over $\E$.   

    Similarly,  ${N}_r(\S)=\{f_a \colon (x,y)\mapsto a\cdot (x,y) \mid a \in \D\}$, so with the same argumentation as above, we get that $X\mathcal{N}_r(\S)X^{-1}=\mathcal{N}_r(\S)$ implies that $X$ is semilinear over $\D$.
\end{proof}
As mentioned before, the pre-semifields we are interested in have generally small nuclei, so this result is in general not particularly useful~--- if $\D\cong \F_p$ then we get for instance no information on $X$ and only little information on $Y$. However, the theorem will still be useful in the case that the nuclei are large. 

Since the pre-semifield $\S_{q,A,B}$ is defined on $\M \times \M$, it is convenient to adopt the following notations for automorphisms
Let $(X,Y) \in \GL(\M \times \M)^2$. With Theorem~\ref{thm:semilinear} in mind, we will always assume that $X,Y$ are semilinear over $\D$ and $\E$, respectively. Then write 
\[X=\begin{bmatrix}
        f_1&f_2\\f_3&f_4
    \end{bmatrix} \text{ and }Y=\begin{bmatrix}
        g_1&g_2\\g_3&g_4
    \end{bmatrix},\]
    where the $f_i$ and $g_i$ are elements in $\End_{\F_p}(\M)$. We then have that $(X,Y) \in \Aut(\mathcal{C}(\S_{q,A,B}))$ if and only if there exists some bijection $(u,v)\mapsto (w,t)$ from $\mathbb{M}^{2}\to\mathbb{M}^2$ such that
    \begin{equation} \label{eq:firstcomp}
         f_1\circ R_{u,v}^{(1)} + f_2\circ R_{u,v}^{(2)} = R^{(1)}_{w,t}\circ (g_1(x)+g_2(y), g_3(x)+g_4(y))
    \end{equation}
        and
    \begin{equation} \label{eq:secondcomp}
         f_3\circ R_{u,v}^{(1)} + f_4\circ R_{u,v}^{(2)} = R^{(2)}_{w,t}\circ (g_1(x)+g_2(y), g_3(x)+g_4(y)).
    \end{equation}
   As mentioned earlier, $X$ and $Y$ are always semilinear over $\D$ and $\E$, respectively, so that $X$ and its subfunctions $f_i$ are necessarily semilinear over $\D$, and $Y$ and its subfunctions $g_i$ are semilinear over $\E$ with same associated field isomorphism.
We will write the $f_i$ and $g_i$ as linearized polynomials, i.e., they are mappings defined by polynomials of the form $\sum_{i=0}^{m-1} a_ix^{p^i}$ with $a_i \in \M$.\\

\textbf{A roadmap of the proof of Theorem~\ref{thm:main}.}
In the end, it turns out that all automorphisms necessarily satisfy that $f_i,g_i$ are all defined by monomials of the same degree (implying that $X,Y \in \Gamma L(2,\M)$), and $X$ is \emph{diagonal}, i.e., $f_2=f_3=0$. We consider different $X,Y$ section by section:
\begin{enumerate}
    \item All $f_i,g_i$ are defined by monomials and $X$ is in diagonal form, i.e., $f_2=f_3=0$. This case is handled in Section~\ref{s:monomials}, and it is here where we will find all automorphisms.
    \item Automorphisms with $X$ in diagonal form, i.e., $f_2=f_3=0$. In this section, we proof that the $f_i,g_i$ of these automophisms are necessarily defined through monomials, reducing the case to the previous case. This will be done in Section~\ref{sec:diagonal}.
    \item The general case. Here we show that all automorphisms necessarily have $X$ in diagonal form, reducing it again to the previous case. This will be done in Section~\ref{sec:general}.
\end{enumerate}

\section{The monomial case} \label{s:monomials}
In this section, we consider $(X,Y)\in \Aut(\mathcal{C}(\S_{q,A,B}))$ such that $f_2=f_3=0$ and all other $f_i,g_i$ are monomials. By simple degree considerations, it is immediate that all monomials necessarily have the same degree. So we can set $f_1=a_1x^{p^i}$, $f_4=d_1x^{p^i}$, $g_1=a_2x^{p^i}$, $g_2=b_2x^{p^i}$, $g_3=c_2x^{p^i}$, $g_4=d_2x^{p^i}$. 

Note that his means that we consider exactly the case where $X,Y \in \Gamma L(2,\M)$.

The first proposition shows that then either $Y$ is also in diagonal form, or in anti-diagonal form.
\begin{proposition} \label{prop:monomial}
    Suppose that for some $0\le i < m$ and $a_1,d_4,a_2,b_2,c_2,d_2\in\mathbb{M}$, for all $u,v\in\mathbb{M}$, there exist some $w,t\in\mathbb{M}$ such that
    \[a_1 x^{p^i}\circ R^{(1)}_{u,v} = R^{(1)}_{w,t}\circ (a_2 x^{p^i} + b_2 y^{p^i}, c_2 x^{p^i} + d_2 y^{p^i})\]
        and
    \[d_1 x^{p^i}\circ R^{(2)}_{u,v} = R^{(2)}_{w,t}\circ (a_2 x^{p^i} + b_2 y^{p^i}, c_2 x^{p^i} + d_2 y^{p^i}).\]

    Then, either $a_2=d_2=0$ or $b_2=c_2=0$.
\end{proposition}

\begin{proof}

By inspecting the monomials $x^{p^i}$ and $x^{rp^i}$ in Eq.~\eqref{eq:secondcomp}, we deduce that for all $u,v\in\mathbb{M}$, there exist some $w,t\in\mathbb{M}$ such that:
\[d_1 A^{p^i} v^{rp^i} = A a_2 t^r + c_2 w^r \text{ and }d_1 v^{p^i}=a_2^r t + Ac_2^r w . \]

Then, consider $v=0$. For all $u\in\mathbb{M}$ there exist some $w,t\in\mathbb{M}$ such that $a_2^{r^2} t^r + A^r c_2^{r^2} w^r = 0$ and $Aa_2 t^r + c_2 w^r = 0$. So,
\[\begin{bmatrix}
    a_2^{r^2} & A^r c_2^{r^2}\\
    A a_2 & c_2
\end{bmatrix}\begin{bmatrix}
    t^r\\w^r
\end{bmatrix}=\mathbf{0}.\]

So either $t^r = w^r = 0$ or $a_2^{r^2} c_2 - A^{r+1} a_2 c_2^{r^2} = 0$. Note that the former is impossible as we require that $(u,v)\mapsto (w,t)$ is a bijection and $(0,0)$ is clearly mapped to $(0,0)$. Now, suppose that $a_2$ and $c_2$ are both nonzero. Then, $A^{r+1}={(\frac{a_2}{c_2})}^{r^2-1}$ which is impossible since $A$ is a non-square. We conclude $a_2c_2=0$.

Similarly, inspecting the monomials $y^{p^i}$ and $y^{rp^i}$ in Eq.~\eqref{eq:secondcomp} leads to
\[d_1 u^{rp^i} = A b_2 t^r + d_2 w^r \text{ and }d_1 A^{p^i} u^{p^i}=b_2^r t + Ad_2^r w . \]
Now, setting $u=0$ leads to the same equations as before, with $b_2,d_2$ replacing $a_2,c_2$, so it leads to $b_2d_2=0$. Note that we cannot have $a_2=b_2=0$ or $c_2=d_2=0$ since  $Y$ is invertible, so either $a_2=d_2=0$ or $b_2=c_2=0$.
\end{proof}

Now, we find all automorphisms where $(X,Y) \in \Gamma L(2,\M)$ and $X,Y$ are both in diagonal form.

\begin{theorem} \label{thm:diagonalcase}
    Suppose we have $(X,Y) \in \Aut(\mathcal{C}_{\S_{q,A,B}}$, where
    \[X=\begin{bmatrix}
        a_1x^{p^i}&0\\0&d_1x^{p^i}
    \end{bmatrix} \text{ and }Y=\begin{bmatrix}
        a_2x^{p^i}&0\\0&d_2x^{p^i}
    \end{bmatrix} ,\]
    for some $0\le i < m$ and $a_1,d_1,a_2,d_2\in\M^\times$.
    Then there exist $\alpha\in\M,\delta\in\D$ such that 
    \begin{equation} \label{eq:cond_1}
        \alpha^{r-1}=A^{p^i-1}, \text{ and }\alpha^{q+1}\delta=B^{p^i-1}
    \end{equation}
    We can write $\delta=\gamma^{p^d+1}\varepsilon^2$ for $\gamma\in \E$, $\varepsilon \in \D$ in $(p^e-1)$ ways, and we get the following conditions:
    $$a_1=d_2^{q+1}\gamma^{r+1}\varepsilon/\alpha, \;\; d_1 = d_2^{r+1}\varepsilon\gamma^{r+1}/\alpha^r,\;\;a_2=d_2\gamma^r\varepsilon/\alpha,\;\;d_2 \in \M^\times.$$

    In total, there are $(p^m-1)(p^e-1)$ such automorphisms for each $i$ that satisfies the conditions in Eq.~\eqref{eq:cond_1}.
\end{theorem}

\begin{proof}
Assume that for $X,Y$ as above Eqs.~\eqref{eq:firstcomp} and~\eqref{eq:secondcomp} hold. Let $u=v=1$. Then, by comparing the coefficients of $x^{qp^i},x^{p^i},y^{qp^i},y^{p^i}$ in Eq.~\eqref{eq:firstcomp} and $x^{rp^i},x^{p^i},y^{rp^i},y^{p^i}$ in Eq.~\eqref{eq:secondcomp}, we get that there exist some $w,t\in\mathbb{M}$ such that the following all hold:
\begin{enumerate}
    \item[(a)] $a_1 = a_2^q w$, $a_1 = a_2 w^q$, $a_1 B^{p^i} = d_2^q Bt$ and $a_1 B^{p^i} = d_2 B t^q$, 
    \item[(b)] $d_1 = a_2^r t$, $d_1 A^{p^i}=a_2 A t^r$, $d_1 A^{p^i} = d_2^r A w$ and $d_1 = d_2 w^r$. 
\end{enumerate}
Here (a) corresponds to the conditions of Eq.~\eqref{eq:firstcomp} and (b) to the conditions derived from Eq.~\eqref{eq:secondcomp}.
Solving for $w$ and $t$ leads to $w=a_1/(a_2^q)$ and $t=d_1/(a_2^r)$, then we can eliminate $w$ and $t$ from all the other equations, leading for the equations in (a) to
\begin{align}
	a_1^{q-1} &= a_2^{q^2-1}, \label{eq:F4_l1}\\
	B^{p^i-1}&=\frac{d_2^qd_1}{a_1a_2^r}, \label{eq:F4_l2}\\	
	B^{p^i-1} &= \frac{d_2d_1^q}{a_1a_2^{Q}}.\label{eq:F4_l3}
\end{align}
From Eq.~\eqref{eq:F4_l1}, we deduce $a_1=a_2^{q+1}\gamma$ for some $\gamma \in \E$. Eqs.~\eqref{eq:F4_l2} and~\eqref{eq:F4_l3} leads to $d_1^{q-1}=a_2^{(q-1)r}d_2^{q-1}$, so $d_1=a_2^rd_2 \gamma_2$ for $\gamma_2 \in \E$. Now, plugging these into the equations from (b) above leads to

\begin{align}
	A^{p^i-1}&=\left(\frac{d_2\gamma_2}{a_2}\right)^{r-1}, \label{eq:F4_r1}\\	
	A^{p^i-1} &= \frac{(a_2^rd_2\gamma_2)^{r-1}}{a_2^{r^2-1}}=\left(\frac{d_2\gamma_2}{a_2}\right)^{r-1}\label{eq:F4_r2},\\
	a_2^{r}d_2 \gamma_2 &= d_2a_2^r\gamma^r\label{eq:F4_r3}.
\end{align}
We conclude that $\gamma_2=\gamma^r$. It remains to verify when Eqs.~\eqref{eq:F4_l2} and~\eqref{eq:F4_r1} are satisfied. Clearly,  it is necessary that $A^{p^i-1}$ is a $(r-1)$-st power, say $A^{p^i-1}=\alpha^{r-1}$. Then, from Eq.~\eqref{eq:F4_r1}, $a_2=d_2\gamma^r\varepsilon/\alpha$, where $\varepsilon \in \D$. 

Eliminating $a_1,d_1,a_2$ in Eq.~\eqref{eq:F4_l2} then leads to 
$$B^{p^i-1}=\left(\frac{\alpha}{\gamma^r \varepsilon}\right)^{q+1}\cdot\gamma^{r-1}=\frac{\alpha^{q+1}}{\gamma^{Q+1}\varepsilon^2}=\frac{\alpha^{q+1}}{\gamma^{p^d+1}\varepsilon^2}.$$
Note that if $\gamma$ and $\varepsilon$ range over $\E^\times$ and $\D^\times$, respectively, the product $\gamma^{p^d+1}\varepsilon^2$ ranges exactly over the entire of $\D^\times$ exactly $(p^e-1)$-times. Indeed, if $d \in \D^\times$ is an arbitrary non-square, we have $(p^e-1)/2$ choices for $\gamma$ such that $\gamma^{p^d+1}$ is an arbitrary non-square, and then $2$ choices for $\varepsilon$, such that $\gamma^{p^d+1}\varepsilon^2=d$. Similarly, if $d \in \D$ is a square, we have  $(p^e-1)/2$ choices for $\gamma$ such that $\gamma^{p^d+1}$ is an arbitrary square, and then again $2$ choices for $\varepsilon$. 

The conditions thus exactly yield the conditions stated in the theorem. It is elementary to check that each of these choices do indeed yield automorphisms.

\end{proof}

We now check the case $(X,Y) \in \Gamma L (2,\M)$, where $X$ is in diagonal form and $Y$ is in antidiagonal form. Together with Proposition~\ref{prop:monomial} and Theorem~\ref{thm:diagonalcase}, this concludes the case where the subfunctions of $X$ and $Y$ are monomials.

\begin{theorem}\label{thm:antidiagonalcase}
      Suppose we have $(X,Y) \in \Aut(\mathcal{C}_{\S_{q,A,B}}$, where
    \[X=\begin{bmatrix}
        a_1x^{p^i}&0\\0&d_1x^{p^i}
    \end{bmatrix} \text{ and }Y=\begin{bmatrix}
        0&b_2x^{p^i}\\c_2x^{p^i}&0
    \end{bmatrix} ,\]
    for some $0\le i < m$ and $a_1,d_1,b_2,c_2\in\M^\times$.
    Then there exist $\alpha\in\M,\delta\in\D$ such that 
    \begin{equation} \label{eq:cond_2}
     \alpha^{r-1}=A^{p^i+1}, \text{ and }\alpha^{q+1}\delta=B^{p^i+1}.
     \end{equation}
     We can write $\delta=\gamma^{p^d+1}\varepsilon^2$ for $\gamma\in \E$, $\varepsilon \in \D$ in $(p^e-1)$ ways, and we get the following conditions:

    In total, there are $(p^m-1)(p^e-1)$ such automorphisms for each $i$ that satisfies the conditions in Eq.~\eqref{eq:cond_2}.
\end{theorem}
\begin{proof}
    Just like in the proof of Theorem~\ref{thm:diagonalcase}, we check the coefficients in Eqs.~\eqref{eq:firstcomp} and~\eqref{eq:secondcomp} of all monomials for $u=v=1$. We obtain that there exist some $w,t\in\mathbb{M}$ such that the following all hold:
\begin{enumerate}
    \item[(a)] $a_1=Bc_2^qt$, $a_1=B_2t^q$, $a_1B^{p^i}=b_2^qw$, $a_1B^{p^i}=b_2w^q$, 
    \item[(b)] $d_1 = Ac_2^rw$, $d_1 A^{p^i}=c_2w^r$, $d_1 A^{p^i} = b_2^rt$ and $d_1 = Ab_2 t^r$. 
\end{enumerate}
This yields from the first equation in (a) and (b) that $t=a_1/(Bc_2^q)$ and $w=d_1/(Ac_2^r)$ and we can eliminate $t,w$ from the equations above. For the second equation from (a) this leads to $a_1^{q-1}= B^{q-1}c_2^{q^2-1}$, so $$a_1=\frac{Bc_2^{q+1}}{\gamma}$$ for some $\gamma \in \E$. The second equation in (b) similarly leads to $d_1^{r-1}=c_2^{r^2-1}A^{p^i+r}$. This implies that $A^{p^i+1}$ is a $(r-1)$-th power, say $A^{p^i+1}=\alpha^{r-1}$. Then $$d_1=c_2^{r+1}\alpha A \varepsilon$$ for some $\varepsilon \in \D.$ Eliminating $t$ and $d_1$ from the last equation in (b) then leads to $$b_2=c_2\alpha \varepsilon\gamma^r.$$
With this, $c_2$ determines $a_1,d_1,b_2$ up to choices for $\varepsilon,\gamma$. We only have to check if any of the other equations in (a) and (b) yield to additional constraints. Both the third and the fourth condition in (a) leads then to 
$$B^{p^i+1}=\alpha^{q+1}\varepsilon^2\gamma^{p^d+1}.$$
As in Theorem~\ref{thm:diagonalcase}, the expression $\varepsilon^2\gamma^{p^d+1}$ ranges over all elements  $\delta \in \D$ exactly $(p^e-1)$ times. So there necessarily exists a $\delta \in \D$ such that $B^{p^i+1}=\alpha^{q+1} \delta.$ The last remaining condition, the third in (b) leads to no additional constraints and the result follows.
\end{proof}

\section{The diagonal case} \label{sec:diagonal}
We now consider the case of automorphisms $(X,Y) \in \Aut(\mathcal{C}(\S_{q,A,B}))$ where $X$ is in diagonal form. The goal in this section is to prove that for these automorphisms, the subfunctions are necessarily monomials, reducing it to the monomial case already considered in Section~\ref{s:monomials}. We start with some lemmas.
 \begin{lemma} \label{lem:binomial}
     Let $f \colon \M \rightarrow \M$ be defined by $f(x)=x u^q + x^q u$ for some $u \in \mathbb{M}^\times$. Then $f$ is invertible.
 \end{lemma}
 \begin{proof}
    Since $f$ is linear over $\E$, it is enough to check that its kernel is trivial. Now $f(x)=xu^q+x^qu=0$ simplifies to $x^{q-1}=-u^{q-1}$ for nonzero $x$, and the result follows since $-1$ is not a $(q-1)$-st power by Lemma~\ref{lem_thm} (i). 
\end{proof}
% \begin{lemma} \label{lem:binomial}
%     Let $f \colon \M \rightarrow \M$ be defined by $f(x)=x u^q + x^q u$ for some $u \in \mathbb{M}^\times$. Then $f$ is invertible and its inverse $f^{-1}$ is defined by the polynomial $h(x)=\sum_{i=0}^{m/e - 1} h_i x^{q^i}$, where \lukas{do we ever need this}
%     \[h_i = \frac{{(-1)}^i}{2} u^{1-q^i(q+1)}.\]
% \end{lemma}
% \begin{proof}
% We check directly if $f\circ h=x$. The polynomial associated to $f \circ h$ is
% $\sum_{i=0}^{m/e-1}(h_iu^q+h_{i-1}^qu)x^{q^i}$, where indices are viewed modulo $m/e$.
% So $f\circ h=x$ if for all $0\le i < m/e - 1$,
% \[h_{i+1} u^q + h_i^q u = 0 \text{ and } h_0 u^q + h_{m/e - 1}^q u = 1.\]

% Now, we see that for $0\le i < m/e - 1$
% \[h_{i+1} u^q + h_i^q u = \frac{{(-1)}^{i+1}}{2}(u^{(1-q^{i+1})(q+1)}-u^{(1-q^{i+1})(q+1)})=0.\]

% Note that $m/e - 1$ is even since $m/e$ is odd. So,
% \[h_0 u^q + h_{m/e-1}^q u = \frac{1}{2}(u^{(q+1)-(q+1)}+u^{(q+1)-q^{m/e}(q+1)})=\frac{1}{2}(1+1)=1.\]
% \end{proof}
\begin{lemma}\label{lemma:bijection}
   Let $f,g\colon \M \rightarrow \M$ such that for each $w\in\mathbb{M}$, there is an $u_w\in\mathbb{M}$ such that
    \[f\circ (x^q u_w + x u_w^q) = (x^q w + x w^q)\circ g.\]
    Then $w\mapsto u_w$ is a bijection or $g=0$.
\end{lemma}

\begin{proof}
Suppose that $u=u_{w_0}=u_{w_1}$. Then,
\[f\circ (x^q u + xu^q)=(x^q w_0 + x w_0^q)\circ g=(x^q w_1 + x w_1^q)\circ g\]

So
\[(x^q (w_0 - w_1) + x{(w_0 - w_1)}^q)\circ g= 0\]

Which means that either $w_0=w_1$ or $g=0$ by Lemma~\ref{lem:binomial}. So, either $g=0$ or $w\mapsto u_w$ is a bijection. 
\end{proof}
\begin{lemma}\label{lemma:qnormal}
  Let $f,g\colon \M \rightarrow \M$ such that $g\ne 0$ and
  \begin{equation} \label{eq:normalize1}
       f\circ (x^q u_w + x u_w^q) = (x^q w + x w^q)\circ g
  \end{equation}

    for all $w\in\mathbb{M}$ and some bijection $w \mapsto u_w$ of $\M$. Then $f$ and $g$ are semilinear over $\E$ with the same associated field automorphism.
\end{lemma}

\begin{proof}
We can write any function on $\M$ in the following way:
Define $f_i$ for $0\le i \le e-1$ such that all $f_i$ are $\E$-linear and 
\[f=\sum_{i=0}^{e-1} x^{p^i} \circ f_i.\]

Similarly, let $g_i$ for $0\le i\le e-1$ such that all $g_i$ are $\E$-linear and
\[g=\sum_{i=0}^{e-1} x^{p^i} \circ g_i.\]

 Then, Eq.~\eqref{eq:normalize1} can be rewritten as
\[x^{p^i}\circ f_i \circ (x^q u_w + x u_w^q) = (x^q w + x w^q)\circ x^{p^i}\circ g_i.\]
 for each $i$ and for all $w\in\mathbb{M}$.
 
For $w\in\mathbb{E}$,
\[(x^q w + x w^q)\circ x^{p^i}\circ g_i = (x^q + x)\circ x^{p^i}\circ g_i\circ w^{p^{-i}}x.\]

So,
\[(x^q w + x w^q)\circ x^{p^i}\circ g_i = x^{p^i}\circ f_i\circ (x^q u_1 + xu_1^q)\circ w^{p^{-i}}x.\]

Then, as $w^{p^{-i}}\in\E$,
\[(x^q w + x w^q)\circ x^{p^i}\circ g_i = x^{p^i}\circ f_i\circ (x^q (u_1 w^{p^{-i}}) + x{(u_1 w^{p^{-i}})}^q).\]

So
\[x^{p^i}\circ f_i\circ (x^q (u_1 w^{p^{-i}}) + x{(u_1 w^{p^{-i}})}^q) = x^{p^i}\circ f_i \circ (x^q u_w + x u_w^q)\]
and we conclude
\[x^{p^i}\circ f_i\circ (x^q (u_1 w^{p^{-i}} - u_w) + x{(u_1 w^{p^{-i}} - u_w)}^q) = 0.\]

Note that either $w^{p^{-i}} u_1 = u_w$ or $f_i=0$ due to invertibility. Next, suppose that for some $i,j$, we have $f_i\ne 0$ and $f_j\ne 0$. Then, $u_w =w^{p^{-i}} u_1=w^{p^{-j}} u_1$ for all $w\in\mathbb{E}$. So, $w^{p^i-p^j}=1$ for all $w\in\mathbb{E}$, meaning that $i\equiv j\pmod{e}$. However, we also know that $0\le i,j\le e-1$, so $i=j$ must be true.

This means that there exists at most one $i$ such that $f_i\ne 0$. Also, for all $i$, if $f_i=0$ then $g_i=0$. So clearly there exists some $i$ such that $g=g_i$ and $f=f_i$, so $f,g$ are semilinear over $\E$ with the same associated automorphism. 
% Note that for this $i$, $u_w=w^{p^{-i}} u_1$ always holds.
\end{proof}

\begin{lemma}\label{lemma:rseparation}
   Let $f \colon \M \rightarrow \M$ be defined by  $f(x)=x^r v + Ax v^r$ for some $v \in \M^\times$. Then $f$ is bijective.
\end{lemma}
\begin{proof}
    $f$ is bijective if and only if its kernel is trivial. We have $f(x)=x^r v + Ax v^r=0$ for nonzero $x \in \M$ if and only if $A=-(x/v)^{r-1}$. But $A$ is a non-square in $\M$, and both $-1$ and $(x/v)^{r-1}$ are squares in $\M$, leading to a contradiction.
\end{proof}

\begin{lemma}\label{lemma:qelim}
     Let $f \colon \M \rightarrow \M$ be semilinear over $\D$. Suppose for all $u \in \M^\times$ we have that $f\circ (x^r u + Ax u^r)$ is $\E$-linear. Then, $f=0$.
\end{lemma}

\begin{proof}

Let $f=\sum_{i=0}^{d-1} f_i\circ x^{p^i}$ for $\D$-linear $f_i$. Then, for all $i\ne 0$,
\[f_i\circ x^{p^i}\circ (x^r u + Ax u^r)=0\]

since this part of $f\circ (x^r u + Ax u^r)$ has no powers in common with any $\E$-linear function. So, $f_i=0$ for all $i \neq 0$ by Lemma~\ref{lemma:rseparation}, so $f$ is $\D$-linear. Let $f=g + h\circ x^r$ for $\E$-linear $g,h$ (recall that $\E$ is a degree $2$ extension of $\D$). Then,
\[g\circ x^r u + h \circ A^r x^r u^{r^2} = 0\]

for all $u\in\mathbb{M}$ by the same logic as above. Then, for $u\in\mathbb{M}^\times$, we can simplify this to
\[g + h\circ A^r u^{r^2-1}x = 0.\]

If $h\ne 0$ then $g+h\circ A^r u^{r^2-1}x$ is not constant, so $h=0$ must be true, and then clearly also $g=f=0$.
\end{proof}

\begin{lemma}\label{lemma:rmonomial}
    Let $f,g \colon \M \rightarrow \M$ such that $f$ is semilinear over $\D$ and $g$ is semilinear over $\E$ and and let $w\mapsto v_w$ be a bijection on $\M$ such that
    \[f\circ (x^r v_w + c_1x v_w^r) = (x^r w + c_2x w^r)\circ g\]
    holds for all $w\in\mathbb{M}$, where $c_1,c_2 \in \M^\times$. Then, $f$ is a  monomial.
\end{lemma}

\begin{proof}
Let us first assume that $f$ and $g$ are linear over $\D$ and $\E$, respectively. We will deal with the general case afterwards.

We have for $w=1$
\[f\circ (x^r v_1 + c_1x v_1^r) = (x^r + c_2x)\circ g.\]

Rewriting the condition in the statement of the lemma, we get
\[(x^r w + c_2xw^r)\circ g=(x^r w + c_2x w + c_2x (w^r - w))\circ g\]

Combining these, we obtain
\begin{equation} \label{eq:intermediate}
   f\circ (x^r v_w + c_1x v_w^r) = wx\circ f\circ (x^r v_1 + c_1x v_1^r) + c_2x(w^r - w)\circ g,
\end{equation}

which means that
\begin{equation} \label{eq:funcelinear}
    f\circ (x^r v_w + c_1x v_w^r) - wx\circ f\circ (x^r v_1 + c_1x v_1^r)
\end{equation}
is $\E$-linear. Recall that $f$ is $\D$-linear, so write $f=\sum_{i=0}^{m/d-1} f_i x^{r^i}$. Similarly, write $f=\sum_{i=0}^{m/d-1} g_i x^{r^i}$, where $g_i=0$ for odd $i$, since $g$ is $\E$-linear and $[\E \colon \D]=2$. Then function in Eq.~\eqref{eq:funcelinear} can only be $\E$-linear if for all odd $i$,
\[f_{i-1} v_w^{r^{i-1}} + f_i c_1^{r^i} v_w^{r^{i+1}} = w(f_{i-1} v_1^{r^{i-1}} + f_i c_1^{r^i} v_1^{r^{i+1}}).\]

We know that $w\mapsto v_w$ is a bijection, so let $v\mapsto w_v$ be the inverse. Then, for all $v\in\mathbb{M}$,
\begin{equation} \label{eq:star}
    f_{i-1} v^{r^{i-1}} + f_i c_1^{r^i} v^{r^{i+1}} = w_v(f_{i-1} v_1^{r^{i-1}} + f_i c_1^{r^i} v_1^{r^{i+1}}).
\end{equation}

Suppose that
\[f_{i-1} v_1^{r^{i-1}} + f_i c_1^{r^i} v_1^{r^{i+1}} = 0.\]

Then,
\[f_{i-1} + f_i c_1^{r^i} v^{r^{i-1}(r^2-1)} = 0\]

for all $v\in\mathbb{M}^\times$ so since $v$ can be chosen arbitrarily, we must have $f_i=f_{i-1}=0$ (note that $r^2-1 \neq 0$ by the conditions on $r$). Now, we know that there is some odd $i$ where either $f_{i-1}\ne 0$ or $f_i\ne 0$ since $f\ne 0$. Call this $i_0$. Then, by Eq.~\eqref{eq:star}, $w_v$ is clearly a polynomial function of $v$ where the only monomials that occur are $v^{r^{i_0+1}}$ and $v^{r^{i_0-1}}$.

Now, suppose that
\[w_v = C_1v^{r^{i_0+1}} + D_1v^{r^{i_0-1}} = C_2v^{r^{i_1+1}} + D_2v^{r^{i_1-1}}.\]

By comparing degrees, we see that either $i_0=i_1$ or two powers are the same and the others have coefficients of $0$. This means that for $i\notin\{ i_0,i_0-1\}$, we have necessarily $f_i=0$. Rewriting Eq.~\eqref{eq:intermediate} in terms of the $f_i,g_i$ leads to
\[f_{i-1}(v^{r^{i-1}} - w_v v_1^{r^{i-1}}) + f_ic_1^{r^i}(v^{r^{i+1}}-w_v v_1^{r^{i+1}})=Ag_i (w_v^r - w_v)\]
for all $i$. 
Since $f_{i_0+1}=0$, we have
\[f_{i_0}(v_w^{r^{i_0}}-w v_1^{r^{i_0}}) = c_2g_{i_0+1}(w^r - w).\]
So either $f_{i_0}=g_{i_0+1}=0$ or ${(\frac{v_w}{v_1})}^{r^{i_0}}=w^r$ for all $w$.
Similarly, since $f_{i_0-2}=0$, we have
\[f_{i_0-1} c_1^{r^{i_0-1}}(v_w^{r^{i_0}}-w v_1^{r^{i_0}})=c_2g_{i_0-1} (w^r - w),\]
and either $f_{i_0-1}=g_{i_0-1}=0$ or ${(\frac{v_w}{v_1})}^{r^{i_0}}=w^r$. Since one of $f_{i_0}$ and $f_{i_0-1}$ are non-zero, we necessarily have that ${(\frac{v_w}{v_1})}^{r^{i_0}}=w^r$. This means that $w_v$ has only powers of $v^{r^{i_0-1}}$ and so only $f_{i_0-1}\ne 0$. Thus, $f$ is a monomial and $\E$-linear, since $i_0-1$ is even.

Now consider the case where $f$ and $g$ are only semilinear over $\D$ and $\E$, respectively. Then we may write $f=x^i \circ f'$ and $g=x^j \circ g'$, where $f',g'$ are linear over $\D$ and $\E$, respectively. Then 
    \[f\circ (x^r v_w + c_1x v_w^r) = (x^r w + c_2x w^r)\circ g\]
    holds if and only if
      \[x^{i-j}\circ f'\circ (x^r v_w + c_1x v_w^r) = (x^r w^{-i} + c_2^{-i}x w^{r-i})\circ g'.\] 
      Since the right hand is $\D$-linear, this means $x^{i-j}$ has to be $\D$-linear as well, so, by our previous result, $x^{i-j}\circ f'$ is a monomial.
\end{proof}

\begin{lemma}\label{lemma:finmonomial}
    Let $f,g \colon \M \rightarrow \M$.  Further, let $w\mapsto v_w$ be a bijection on $\M$ such that
    \[f\circ (x^r v_w + Ax v_w^r) = (x^r w + Ax w^r)\circ g\]
    holds for all $w\in\mathbb{M}$.  Then, $f$ is a monomial if and only if $g$ is  a monomial and both monomials are of the same degree.
\end{lemma}

\begin{proof}
    If $f=x^i$, we have $f\circ (x^r v_w + Ax v_w^r) = (x^r v_w^i + A^ix v_w^{r+i}) \circ g$, so the condition becomes
    \[(x^r v_w^i + A^ix v_w^{r+i}) \circ f = (x^r w + Ax w^r)\circ g,\]
    implying that $g$ is a monomial as well. The other direction works in a similar way.
\end{proof}
Using the same idea, we get the following lemma as well.
\begin{lemma}\label{lemma:othermonomial}
Let $f,g \colon \M \rightarrow \M$ and
    let  further $w\mapsto u_w$ be a bijection on $\M$ such that
    \[f\circ (x^q u_w + x u_w^q) = (x^q w + x w^q)\circ g\]
     Then, $f$ is a monomial if and only if $g$ is a monomial. In that case, both monomials are of the same degree.
\end{lemma}
% \begin{proof}

% Write $f(x)=\sum_{i} f_i x^{q^i}$. Clearly, by the conditions, $g$ has to be $\E$-linear as well, say $g(x)=\sum_{i} g_i x^{q^i}$, where only one $g_i \neq 0$,  since $g$ is a monomial, say $g_{i_0}\neq 0$. The condition then yields
% \[f_{i-1} u_w^{q^{i-1}} + f_i u_w^{q^{i+1}} = g_{i-1}^q w + g_i w^q\]

% for all $i$. Now, suppose that $i\ne i_0$. Then, either $i-1\ne i_0$ or $i+1\ne i_0$ is also true. So, for all $i$ with $i\notin \{i_0,i_0+1\}$ we have that $g_{i}=g_{i-1}=0$, so
% \[f_{i_1-1} u_w^{q^{i-1}} + f_{i_1} u_w^{q^{i+1}} = 0\]

% for all $w$. Since $w\mapsto u_w$ is a bijection, this must hold for all $u\in\mathbb{M}^\times$, which is obviously only the case if $f_{i}=f_{i_1}=0$. So, clearly $f_i=0$ for all $i\ne i_0$. Then, we see that $f$ is a monomial of the same degree as $g$.
% \end{proof}

We are now ready to state the main result of this section that states that any automorphism $(X,Y) \in \Aut(\mathcal{C}(\S_{q,A,B}))$ where $X$ is in diagonal form necessarily has only monomials $f_i,g_i$ as subfunctions.

\begin{theorem} \label{thm:monomials}

    Suppose that $(X,Y) \in \Aut(\mathcal{C}(\S_{q,A,B}))$ where $X$ is in diagonal form. Then, $f_1,f_4,g_1,g_2,g_3,g_4$ are all monomials.
\end{theorem}

\begin{proof}
    If $(X,Y) \in \Aut(\mathcal{C}(\S_{q,A,B}))$  then there exists for all $u,v\in\mathbb{M}$  some $w,t\in\mathbb{M}$ such that
    \[f_1\circ R^{(1)}_{u,v} = R^{(1)}_{w,t}\circ (g_1(x)+g_2(y),g_3(x)+g_4(y))\]
    and
    \[f_4\circ R^{(2)}_{u,v} =R^{(2)}_{w,t}\circ (g_1(x)+g_2(y),g_3(x)+g_4(y)).\]
So, we see that for all $u,v$ there exist some $w,t$ such that the following are all true:
\begin{align}
    f_1\circ (x^q u + x u^q) = (x^q w + x w^q)\circ g_1 + B(x^q t + x t^q)\circ g_3 \label{eq:al1}\\
    f_1\circ B(y^q v + y v^q) = (y^q w + y w^q)\circ g_2 + B(y^q t + y t^q)\circ g_4 \label{eq:al2}
\end{align}

Suppose that $w=0$. Then, for all $t$ there exists some $u$ such that
\[f_1\circ (x^q u + x u^q) = B(x^q t + x t^q)\circ g_3\]

and so by Lemma~\ref{lemma:qnormal} $f_1$ and $g_3$ are semilinear over $\E$ with the same associated field automorphism, or $g_3=0$. By the same logic, setting $w=0$ or $t=0$ and looking at Eq.~\eqref{eq:al1} and~\eqref{eq:al2}, we find that $f_1,g_1,g_2,g_3,g_4$ are all semilinear over $\E$ with the same associated field automorphism, or are the zero mapping.

Likewise, by Lemma~\ref{lemma:bijection} either $t\mapsto u_t$ for $w=0$ is a bijection or $g_3=0$. Similarly, considering Eq.~\eqref{eq:al2}, we get that $t\mapsto u_t$ for $w=0$ is a bijection or $g_4=0$. Since $g_3$ and $g_4$ cannot both be $0$, we have necessarily that $t\mapsto u_t$ for $w=0$ is a bijection.

Next, we see that for all $u,v$ there exist some $w,t$ such that:
\begin{align}
    f_4\circ (x^r v + Ax v^r) = (x^r t + Ax t^r)\circ g_1 + (Ax^r w + x w^r)\circ g_3\label{eq:al3}\\
    f_4\circ (Ay^r u + y u^r) = (y^r t + Ay t^r)\circ g_2 + (Ay^r w + y w^r)\circ g_4\label{eq:al4}
\end{align}
We proceed similarly to before. Let $w=0$. By Eqs.~\eqref{eq:al2} and~\eqref{eq:al3}, we have that for all $t$ there exists some $v$ such that
\[f_1\circ B(y^q v + y v^q) =  B(y^q t + y t^q)\circ g_4\text{ and } f_4\circ (x^r v + Ax v^r) = (x^r t + Ax t^r)\circ g_1.\]

Then, by Lemma~\ref{lemma:bijection} either $g_4=0$ or $t\mapsto v$ is a bijection. If the former is true then (since $f_1\neq 0$) we have $v=0$ for all $(t,w)$ where $w=0$. In this case, we see that $(x^r t + Ax t^r)\circ g_1=0$ for all $t$ so $g_1=0$ by Lemma~\ref{lemma:rseparation}.

If $g_4 \neq 0$, which is equivalent to $g_1 \neq 0$, then by  Lemma~\ref{lemma:rmonomial} we see that  $f_4$ is a monomial and by Lemma~\ref{lemma:finmonomial}, $g_1$ is a monomial as well.

We do likewise for all nonzero $g_i$ by picking different combinations from Eqs.~\eqref{eq:al1} to~\eqref{eq:al4}, and find that all nonzero $g_i$ and $f_1$ are monomials as well by invoking Lemmas~\ref{lemma:finmonomial} and~\ref{lemma:othermonomial}. 

If $g_1=g_4=0$ then $g_2,g_3$ are necessarily non-zero and we can again follow the same arguments as before by picking the correct combinations from Eqs.~\eqref{eq:al1} to~\eqref{eq:al4} to conclude that $f_1,f_4,g_2,g_3$ are monomials.
\end{proof}

\section{The general case} \label{sec:general}
In this section, we deal with the general case, i.e., we place no further restrictions on $(X,Y)\in \Aut(\mathcal{C}(\S_{q,A,B}))$ except that $X$ is semilinear over $\D$ and $Y$ is semilinear over $\E$. Our goal is to show that $X$ is necessarily in diagonal form, reducing it to the case investigated in the previous section.

\begin{proposition} \label{prop:f20}
    Suppose $(X,Y)\in \Aut(\mathcal{C}(\S_{q,A,B}))$ is an automorphism. Then $f_2=0$.
\end{proposition}

\begin{proof}
Consider the monomials in $x$ of Eq.~\eqref{eq:firstcomp} for $u=0$. Then, for all $v\in\mathbb{M}$,
\[f_2\circ (x^r v + Ax v^r) = (x^q w + x w^q)\circ g_1 + B(x^q t + x t^q)\circ g_3\]

holds for some $w,t\in\mathbb{M}$. We know that $g_1$ and $g_3$ are both semilinear over $\E$ with the same associated field automorphism since $Y$ is semilinear over $\E$. So we can write $g_1=x^{p^i} \circ h_1$ and $g_3=x^{p^i}\circ h_3$ for $\E$-linear mappings $h_1$ and $h_3$, and some $i$ where $0\le i\le e-1$. 

So, for all $v\in\mathbb{M}$ there exist some $w,t\in\mathbb{M}$ such that
\[x^{p^{-i}}\circ f_2\circ (x^r v + Axv^r) = (x^q w + x w^q)\circ h_1 + B(x^q t + x t^q)\circ h_3\]

And by Lemma~\ref{lemma:qelim}, $x^{p^{-i}}\circ f_2 = 0$ and so $f_2=0$.
\end{proof}

\begin{lemma}\label{lemma:q3elim}
    Suppose that for all $u \in \M$ there exist some $w,t \in \M$ such that
    \[(e + f\circ x^r)\circ (x^q u + x u^q) = (x^r t + Ax t^r)\circ g + (Ax^r w + xw^r)\circ h\]

    holds for $\E$-linear functions $e,f,g,h$. Then, $h=0$.
\end{lemma}

\begin{proof}
We write all $\E$-linear functions as linearized polynomials, e.g. $e = \sum e_ix^{q^i}$ and the same for $f,g,h$ which then have coefficients $f_i,g_i,h_i$. Comparing the monomials, 
we see that for all $i$,
\begin{equation} \label{eq:firstpart}
    e_{i-1} u^{q^{i-1}} + e_i u^{q^{i+1}} = At^r g_i + w^r h_i
\end{equation}

and
\[f_{i-1} u^{rq^{i-1}} + f_i u^{rq^{i+1}} = tg_i^r +Aw h_i^r\]

must hold. Taking the latter equation to the $r$-th power yields
\begin{equation} \label{eq:secondpart}
    f_{i-1}^r u^{q^{i+1}} + f_i^r u^{q^{i+3}} = t^r g_i^{q^2} +A^r w^r h_i^{q^2}.
\end{equation}

We now compare the powers of the monomials in Eqs.~\eqref{eq:firstpart} and~\eqref{eq:secondpart}. We  have to distinguish (for each $i$) the following different cases that can occur: \\
\noindent\textit{Case 1:} $g_i=0$ and $h_i=0$: Then clearly $e_{i-1}=e_i=f_{i-1}=f_i=0$ is necessary.

\noindent\textit{Case 2:} $g_i=0$ and $h_i\ne 0$: Then, $w^r=Cu^{q^{i+1}}$ for some constant $C$ and $e_{i-1}=f_{i}=0$.

\noindent\textit{Case 3:} $g_i\ne 0$ and $h_i=0$: Similarly, $t^r=Cu^{q^{i+1}}$ for some constant $C$ and $e_{i-1}=f_{i}=0$.

\noindent\textit{Case 4:} $g_i\ne 0$ and $h_i\ne 0$: We first check if Eqs.~\eqref{eq:firstpart} and~\eqref{eq:secondpart} are linearly dependent. If they are,then

\[\frac{g_i^{q^2-1}}{A}=A^r h_i^{q^2-1},\]

or equivalently
\[{\left(\frac{g_i}{h_i}\right)}^{r^2-1}=A^{r+1}.\]

So $A$ must be a square, which is impossible, so we know that Eqs.~\eqref{eq:firstpart} and~\eqref{eq:secondpart} are linearly independent. Then, clearly, $t^r$ and $w^r$ can both be written as trinomials $C_1u^{q^{i-1}}+C_2 u^{q^{i+1}}+C_3 u^{q^{i+3}}$. Here, we consider $t,w$ as linearized polynomials in $u$.\\
We now show that $e_i=f_i=0$ for all indices $i$, implying $e=f=0$.

If  Case 1 holds for some index $i_0$ then $e_{i_0}=f_{i_0}=0$ is immediate. 

Assume now Case 2 holds for some index $i_0$, so $w^r=Cu^{q^{i_0+1}}$ for some constant $C$ and $f_{i_0}=0$. Then we have Case 1 or Case 3 for $i_0+1$ (note for instance that Case 4 is impossible since $Cu^{q^{i_0+1}}$ is not of the form described in Case 4). In both cases $e_{i_0+1-1}=e{i_0}=0$, so we have $e_{i_0}=f_{i_0}=0$ as claimed.

If Case 3 holds for some index $i_0$, an identical argument as in Case 2 yields $e_{i_0}=f_{i_0}=0$.

If Case 4 occurs for some index $i_0$, this immediately leads to Case 1 for index $i_0+1$, which also leads to $e_{i_0}=f_{i_0}=0$. We conclude that $e_i=f_i=0$ for all indices $i$.
\end{proof}
% \begin{lemma}\label{lemma:otherelim2}
%     Suppose that $(f,g)$ is an automorphism and $f_2=0$. Then, $f_3=0$ as well.
% \end{lemma}

% \noindent\textbf{Proof:}

% Consider $w=0$. Since $f_2=0$, for all $t\in\mathbb{M}$ there exists some $u\in\mathbb{M}$ such that
% \[B(x^q t + xt^q)\circ g_3 = f_1\circ (x^q u + xu^q)\]

% Which means that $g_3$ normalizes $\mathbb{F}_q$ in the same way as $f_1$, or $g_3=0$ and $u=0$ for all $t$ by Lemma~\ref{lemma:qnormal}. The same holds for $g_1$, $g_2$ and $g_4$, so all are either 0 or normalize $\mathbb{F}_q$ in the same way as $f_1$, and so in the same way as each other.

% Next, consider the $v=0$ case where for all $u$ there exist some $w,t$ such that
% \[f_3\circ (x^q u + xu^q) = (x^r t + Ax t^r)\circ g_1 + (Ax^r w + x w^r)\circ g_3\]

% where $g_1,g_3$ normalize $\mathbb{F}_q$ in the same way, so this is essentially the $\mathbb{F}_q$-linear case for those, and so $f_3=0$ by Lemma~\ref{lemma:q3elim}.
\begin{proposition}\label{prop:f30}
    Suppose $(X,Y)\in \Aut(\mathcal{C}(\S_{q,A,B}))$ is an automorphism. Then $f_3=0$.
\end{proposition}

\begin{proof}
Consider the monomials in $x$ of Eq.~\eqref{eq:secondcomp} for $v=0$. Then, for all $u\in\mathbb{M}$,
\begin{equation} \label{eq:semilinear}
    f_3\circ (x^q u + x u^q) = (x^r t + Ax t^r)\circ g_1 + (Ax^r w + x w^r)\circ g_3
\end{equation}

holds for some $w,t \in \M$. Note that $g_1,g_3$ are semilinear over $\E$ with same associated field automorphism, say $x \mapsto x^{p^i}$ and $f_3$ is semilinear over $\D$. So we can rewrite Eq.~\eqref{eq:semilinear} as
\begin{equation} \label{eq:lasteq}
    f_3\circ x^{p^{-i}} \circ (x^q u + x u^q) = (x^r t + Ax t^r)\circ g_1' + (Ax^r w + x w^r)\circ g_3'
\end{equation}

where $g_1'$ and $g_3'$ are the $\E$-linear mappings associated to the semilinear mappings $g_1,g_3$. Note that the right hand side of Eq.~\eqref{eq:lasteq} is $\D$-linear, so $f_3\circ x^{p^{-i}}$ is necessarily $\D$-linear as well. Since $\E$ is a degree 2 extension of $\D$, we can write $f_3\circ x^{p^{-i}}=e + f\circ x^r$ where $e,f$ are $\E$-linear. By Lemma~\ref{lemma:q3elim}, we get $e=f=f_3=0$, proving our claim.
\end{proof}

Propositions~\ref{prop:f20} and~\ref{prop:f30} together thus show that $X$ is in diagonal form which was the goal of this section.
\section{Putting everything together: The proof of Theorem~\ref{thm:main}}
We are now ready to prove the main theorem of this work, Theorem~\ref{thm:main}.

\begin{proof}[Proof of Theorem~\ref{thm:main}]
    Let $(X,Y)\in \Aut(\mathcal{C}(\S_{q,A,B}))$. 
    By Propositions~\ref{prop:f20} and~\ref{prop:f30}, $X$ is in diagonal form. Then by Theorem~\ref{thm:monomials}, all subfunctions are monomials. By Proposition~\ref{prop:monomial}, $Y$ is then in diagonal or anti-diagonal form. These two cases are handled in Theorems~\ref{thm:diagonalcase} and~\ref{thm:antidiagonalcase}. Note that if $A^2$ is not an $(r-1)$-st power, then $A^{p^i\pm1}$ also cannot be an $(r-1)$-st power, except $A^{p^0-1}$. So in this case, all autotopisms are from Theorem~\ref{thm:diagonalcase} with $i=0$. If $A^2$ is an $(r-1)$-st power, then $A^{p^i+1}$ is an $(r-1)-st$ power if and only if $A^{p^i-1}$ is, so we can for each admissible $i$ exactly $2(p^m-1)(p^e-1)$ autotopisms, with $(p^m-1)(p^e-1)$ each from Theorem~\ref{thm:diagonalcase} and~\ref{thm:antidiagonalcase}. 
\end{proof}

\section{Conclusion}
We have determined the autotopism group of the family of semifields found in~\cite{golouglu2023exponential}, which is equivalent to the collineation group of the corresponding projective plane as well as the automorphism group of the corresponding rank-metric code/spread set. This, in particular, makes it much easier to decide if a given semifield is isotopic to these semifields since isotopic semifields have conjugate autotopism groups (see e.g.~\cite[Lemma 5.1.]{golouglu2023exponential}). For instance, if we find that some semifield $\S'$ has a subgroup that is not isomorphic to a subgroup of the group described in Theorem~\ref{thm:main}, we know that $\S'$ is not isotopic to the semifields of the G\"olo\u{g}lu-K\"olsch family investigated here. 

As a simple example, the autotopism group of the twisted fields is known~\cite{biliotti1999collineation}, and a simple inspection shows that the respective autotopism groups are non-isomorphic, and as such the generalized twisted fields and the semifields we investigated in this paper are never isotopic. Many similar results of this type can be found. 

We also want to highlight that a long standing conjecture is that autotopism groups of proper finite semifields are always solvable (see~\cite[Chapter VIII., Section 6]{hughesbook}). This is equivalent to the collineation group of semifield planes being solvable (note that this conjecture is false in general for tranlation planes, with for instance the Hall planes giving a counterexample~\cite{hughesbook}). There is not much evidence in favor of this conjecture, outside of some explicit examples of semifields with solvable autotopism group, some structural results listed in~\cite{hughesbook}, and the result that $A_5$ cannot appear as a subgroup of the autotopism group of a semifield~\cite{kravtsova2020alternating}. Theorem~\ref{thm:main} gives a new large family of semifields with known autotopism group, and as such gives more empirical evidence in favor of the conjecture. Indeed, our case is particularly important and interesting because the autotopism group is much larger than the subgroup generated by the nuclei, which was not the case in the previously known examples~\cite{dempwolff2011autotopism, hui2015autotopism}. The nuclei, of course, are always solvable subgroups, so it seems natural to expect possible counterexamples to the conjecture coming from semifields with small nuclei.

For future work, it would be beneficial to explicitly determine the autotopism groups of more semifields to make the isotopism question easier to handle. In particular, some recent large and unifying constructions of semifields have recently been found~\cite{kolsch2024unifying,lobillo2025quotients} and determining their automorphism groups would be very desirable. 

\subsection*{Acknowledgments}
The majority of this work was conducted when the second and third author visited the University of South Florida (USF) for the REU program on Cryptography and Coding Theory at USF in the summer of 2024, funded by NSF grant 2244488. The authors gracefully acknowledge the support. We would like to thank Luca Bastioni for many discussions around the project.

\bibliographystyle{amsplain}
\bibliography{semifields} 
\end{document}